\documentclass[ECP]{ejpecp} % replace ECP by EJP if needed.
\usepackage{enumerate}
\usepackage{upgreek}
\usepackage{bbold}
\usepackage{xcolor}

\newcommand{\R}{\mathbb{R}}
\newcommand{\proba}{\mathbb{P}}

\SHORTTITLE{Marginals of a spherical spin glass model with correlated disorder}

\TITLE{Marginals of a spherical spin glass model \\ with correlated disorder}

\AUTHORS{
 Jean~Barbier\footnote{International Centre for Theoretical Physics, Trieste, Italy. \EMAIL{jbarbier@ictp.it}}
 \and
 Manuel~S\'aenz\footnote{International Centre for Theoretical Physics, Trieste, Italy. Corresponding author. \EMAIL{msaenz@ictp.it}}}%AUTHORS

\AMSSUBJ{82B44; 82D30}

\SUBMITTED{} % Edit.
\ACCEPTED{}

\VOLUME{}
\YEAR{}
\PAPERNUM{}
\DOI{}

\ABSTRACT{In this paper we prove the weak convergence, in a high-temperature phase, of the finite marginals of the Gibbs measure associated to a symmetric spherical spin glass model with correlated couplings towards an explicit asymptotic decoupled measure. We also provide upper bounds for the rate of convergence in terms of the one of the energy per variable. Furthermore, we establish a concentration inequality for bounded functions under a higher temperature condition. These results are exemplified by analysing the asymptotic behaviour of the empirical mean of coordinate-wise functions of samples from the Gibbs measure of the model.}

\begin{document}

\section{Introduction}

The field of spin glasses \cite{mezard1987spin} is a branch of statistical mechanics that started as a way to explain strange magnetic behaviour of disordered materials. More recently, there have also been a wide variety of inference problems that were successfully studied by applying its tools and heuristics \cite{mezard2009information,zdeborova2016statistical}. Some of the strategies used to approach these problems include belief-propagation and approximate message passing algorithms \cite{donoho2013information,kabashima2003cdma}, along with the cavity \cite{aizenman2003extended,coja2018information,mezard1987spin,panchenko2013sherrington,tala2010libro}, interpolation \cite{guerra2002thermodynamic,tala2010libro}, and adaptive interpolation methods \cite{jean2019adaptive,jean2019adaptive2}. All of these allow to establish the asymptotic log-partition function of physical and information processing systems and the performance of Bayesian estimators but rely, in most cases, on the randomness defining the model (the ``quenched disorder'') being a collection of i.i.d. and/or normal random variables.

However, in many applications the disorder present in the system may come from a distribution with some underlying complex structure. This makes inadequate the approaches that require it to be i.i.d./normally distributed. The particular attention that models with i.i.d./normal disorder have received is mainly because many of the tools developed in the field of spin glasses require these types of distributions. Thus, at present, there are ongoing efforts to extend the analysis of these models to systems where the disorder is given by more general \emph{rotationally invariant matrices} \cite{bhattacharya2016high,fan2020AMProt,fan2021ortho,kurchan2021annealed,gabrie2019entropy,gerbelot2020asymptotic,ma2021analysis,maillard2019hightemp,parisi1995mean,rangan2019vector,takahashi2020macroscopic,takeda2006analysis}. This is a family of random matrices that have very general distributions: many classical random matrix ensembles can be thought of as particular cases of them. However, very little is known about the high-dimensional limit of problems involving this type of correlated disorder, especially on the rigorous side.

Here we establish the weak convergence of the finite marginals of the Gibbs measure associated to a symmetric spherical model in the high-temperature phase, with explicit rates of convergence expressed in terms of the concentration rate of the energy per variable. We also prove an upper bound for the variance of functions of the coordinates of the model under a more restrictive high-temperature condition. Finally, we apply these results to study the limit of empirical means of coordinate-wise functions of samples of the model.

One objective of the present work is to apply the cavity method to a disordered system with couplings that are not i.i.d./normally distributed. As far as we know, this is the first example of a rigorous use of the cavity method for a model with such correlated disorder. By doing this, we extend results on the spherical spin glass model which has been extensively studied for normal disorder \cite{belius2019tap,chen2013aizenman,crisanti1993sphericalp,crisanti1992spherical,kosterlitz1976spherical,subag2018free,talagrand2006free}. As discussed in \cite{maillard2019hightemp}, this model is a good play-ground for developing tools to study more complicated systems with rotationally invariant couplings. But the results presented here also have an interest in themselves as this model is also closely related to the large deviation theory of random matrices. Indeed, the free energy limit of the model studied in this paper has been rigorously obtained in \cite{guionnet2005fourier} and dynamical results for a relaxed version of the model are given in \cite{arous2001aging}. In the latter reference the weak limit of the marginals of this relaxed version of the symmetrical spherical model are established. This result differs from ours in that the techniques used are very different and in the fact that they prove it for a simpler model with a \emph{soft spherical constraint}. Finally, owing to its connection with high-dimensional inference problems with non-normal disorder, the model has recently been the object of a renewed interest. In \cite{kurchan2021annealed}, the authors examined the phenomenology of an annealed version of the model, and in \cite{maillard2019hightemp} some high-temperature expansion is provided. See also \cite{bhattacharya2016high} for further rigorous results. Finally, in \cite{fan2020AMProt,fan2021ortho,gerbelot2020asymptotic} related models were studied from an algorithmic perspective.

\section{Description of the model and notation}\label{sec:setting}

Let $N\geq1$. A real symmetric \emph{rotationally invariant matrix} $J\in\R^{N\times N}$ is a random matrix such that $J = O^\intercal D O$ where $O$ is distributed according to the Haar measure over the orthogonal matrices in $\R^{N\times N}$ and $D = \mbox{diag}(\gamma^{(N)}_1,\dots,\gamma^{(N)}_N)\in\R^{N\times N}$. The \emph{symmetric spherical model} is defined by the Hamiltonian $N s^\intercal J s$, with $J$ a quenched rotationally invariant coupling matrix and $s\in S^{N-1}$ a spins vector on the unit sphere of dimension $N$.

From now on we assume the following hypothesis on the eigenvalues $(\gamma^{(N)}_i)_{i\leq N}$.
\begin{hypothesis}[Eigenvalues]\label{hyp:eigenvalues}
    There is some fixed interval $I\subseteq\R$ for which for all $N,i\geq1$ we have $\gamma_i^{(N)}\in I$. Moreover, the measure $N^{-1} \sum_{i\leq N} \delta_{\gamma^{(N)}_i}$ converges in the 1-Wasserstein metric towards a limiting distribution $\rho(\cdot)$.
\end{hypothesis}
To ease the notation, from now on we will omit the superscript $N$ in $\gamma_i^{(N)}$. Notice that we can rotate the original vector $s$ by $O$ to obtain a new spherical vector $s'$ and Hamiltonian $N\sum_i \gamma_i (s_i')^2$. Because the distribution of a uniform vector $s'$ over the sphere $S^{N-1}$ is equal to that of $g/\|g\|$ with $g=(g_1,\dots,g_N)\in\R^N$ a standard normal vector of dimension $N$, we can alternatively define the Hamiltonian according to
\begin{equation}
    H_N(g) :=   N   \sum_{i\leq N} \gamma_i \frac{g_i^2}{\|g\|^2}.
\end{equation}
Because the model is invariant under a constant shift of all the eigenvalues $\gamma_1,\dots,\gamma_N$, we can assume without loss of generality that the eigenvalues $(\gamma_i)_{i\leq N}$ are contained in the non-negative interval $[0,\tilde\gamma]$ for some $\tilde\gamma >0$. We will also set $\langle \cdot \rangle$ as its corresponding Gibbs mean with respect to a standard normal measure on $\R^N$, which we will denote by $G(\cdot)$. That is, for every integrable function $f:\R^N\to\R$ we define
\begin{equation}
    \langle f(g) \rangle := \frac{1}{Z_N}\int_{\R^N}dG(g) f(g) \exp \theta H_N(g)\label{GibbsMean}
\end{equation}
where $\theta\in\R$ is the \emph{inverse temperature} of the model. For simplicity, we will refer to it as the  \emph{temperature} from now on. Here $Z_N$ is a normalising constant which we will refer to as the \emph{partition function}. Whenever we want to emphasise the dependence on the temperature, we will write $\langle\cdot\rangle_\theta$ and $Z_N(\theta)$ instead. We will also define the intensive free energy and energy of the model according to $f_N := N^{-1} \ln Z_N$ and $h_N := N^{-1} H_N$, respectively.

Let $S_\rho : \R\backslash[0,\tilde\gamma] \to \R$ be the Stieltjes transform of the distribution $\rho(\cdot)$ given by
\begin{equation*}
    S_\rho(z) := \int_0^{\tilde\gamma} \frac{\rho(d\gamma)}{z-\gamma}.
\end{equation*}
Define also $S_{max} := \lim_{z\downarrow\tilde\gamma} S_\rho(z)$ and $S_{min} := \lim_{z\uparrow 0} S_\rho(z)$ which can be infinite. Because $S_\rho$ is a bijection between $\R\backslash[0,\tilde\gamma]$ and its image $(S_{min},S_{max})\backslash\{0\}$, it has some inverse 
\begin{equation*}
    K_\rho :(S_{min},S_{max})\backslash\{0\}\to\R\backslash[0,\tilde\gamma].
\end{equation*}
We will then define the \emph{R-transform of} $\rho(\cdot)$ as the function
\begin{equation*}
    R_\rho:(S_{min},S_{max})\backslash\{0\}\to\R\backslash[0,\tilde\gamma] \quad \mbox{given by} \quad R_\rho(x) := K_\rho(x) - 1/x.
\end{equation*}
For any $\theta\in(S_{min}/2,S_{max}/2)\backslash\{0\}$ we  set $v(\theta) := R_\rho(2\theta)$. We will sometimes omit the $\theta$ in the argument and just write $v$. From now on we denote the ``high-temperature region'' $T_\rho:=(S_{min}/2,S_{max}/2)\backslash\{0\}$. Note that for $\theta=0$ the model is anyway trivial.

Throughout the paper, most results will be given in terms of the rate of convergence in $L^2$ of the mean of the intensive energy towards its limit, which we will denote by
\begin{equation}\label{eq:def_aN}
    a_N=a_N(\theta) := \big\langle (h_N - v(\theta))^2\big\rangle_\theta.
\end{equation}
Here the subscript $\theta$ in the mean $\langle\cdot\rangle_\theta$ indicates that it is with respect to the Gibbs measure at temperature $\theta$. The reader should keep in mind that, as we prove in {Lemma\phantom{ }\ref{lem:en_contr}}, in the setting considered $a_N$ is always a vanishing sequence.

\section{Main results and application to sample means}\label{sec:main}

There are two main results. The first one gives the convergence of the means of bounded functions of the vector $g$ towards a measure where a single spin is decoupled from the rest. It also provides an explicit convergence rate as a function of $a_N$. 

To state the result we will first define a new Hamiltonian where the $d$-th spin is decoupled from the rest:
\begin{equation}\label{eq:def_Hd}
    H_d(g) := N \sum_{i\leq N,\ i\neq d} \gamma_i \frac{g_i^2}{\|\bar g\|^2} - ( v(\theta) - \gamma_d) g_d^2,
\end{equation}
where $d\leq N$ is any spin index and we let $\bar g:=(g_1,\dots,g_{d-1},g_{d+1},\dots,g_{N})\in\R^{N-1}$. In other words, we have $H_d(g) = {N}(N-1)^{-1} H_{N-1}(\bar g) - ( v - \gamma_d) g_d^2$. In a similar way as before, we denote by $\langle\cdot\rangle_d$ the expectation with respect to the Gibbs measure associated to this Hamiltonian $H_d$ and at temperature $\theta$ under the standard normal measure $G(\cdot)$ on $\R^N$ (i.e., it is defined by replacing $H_N$ by $H_d$ in \eqref{GibbsMean}). 

\begin{theorem}[Decoupling in the high-temperature phase]\label{thm:asymptotic_mean}
    Let $(\tilde a_N)_{N\geq1}$ be the sequence given by $\tilde a_N(\theta) := 1/\sqrt N+\sqrt{a_N(\theta)}$ and $d\geq1$. For any finite $\theta\in T_\rho$ there exists some constant $K(\theta) > 0$ such that, for every bounded function $f:\R^N\to \R$,
    \begin{equation*}
        \big| \langle f(g) \rangle - \langle f(g) \rangle_d \big| \leq  K(\theta) \|f\|_\infty \tilde a_N(\theta).
    \end{equation*}
\end{theorem}
One important observation is that this result is established for the largest possible interval of values of the temperature $\theta$. Indeed, as discussed in \cite{guionnet2005fourier}, if $\theta\not\in T_\rho$ the coordinate associated to the largest eigenvalue is of order $N$. This implies that, if a limiting marginal exists, it will not have a finite mean.

Note that, given $a_N$, because the dependence of the convergence rate on $f$ is explicit, it can be extended to non-bounded functions with a sufficiently slow growth rate. This can be achieved by a canonical approximation argument, as we do in the example below.

Now, define for every $\gamma\in[0,\tilde \gamma]$ the probability measure $\mu_\gamma(\cdot)$ as the law of a centred normal random variable of variance $\sigma_\gamma^2 := 1/(1+2 \theta( v(\theta) - \gamma))$. We then have the following result that establishes the weak convergence of the finite marginals of finite sets of coordinates from the model. This follows directly from $k$ applications of Theorem \ref{thm:asymptotic_mean}.

\begin{corollary}[Finite marginals]\label{cor:marginals}
    If $\theta\in T_\rho$, then for every $k\geq1$ and $\{d_1,\dots,d_k\}\subseteq [N]$ a subset of $k$ distinct indices we have that for all $f:\R^k\to\R$ bounded there exists a constant $K'(\theta,k)>0$ such that, if $(\tilde a_N)_{N\geq1}$ is as in Theorem \ref{thm:asymptotic_mean}, 
    \begin{equation*}
        \Big|\langle f(g_{d_1},\dots,g_{d_k})\rangle -  \int_{\R^k} d\mu_{\gamma_{d_1}}(x_1)\cdots d\mu_{\gamma_{d_k}}(x_k) f(x_1,\dots,x_k)  \Big|\le K'(\theta,k) \|f\|_\infty \tilde a_N(\theta).
    \end{equation*}
\end{corollary}

\begin{remark}[High-temperature condition]
Notice that if the mass that the measure $\rho(\cdot)$ gives to the balls $B_\varepsilon(0)$ and $B_\varepsilon(\tilde\gamma)$ decays slowly enough when $\varepsilon\to0^+$, then the interval $T_\rho$ will be equal to $\R\backslash\{0\}$. For example, this is easily seen to be the case when $\rho(\cdot)$ is a convex combination of Dirac measures over different points. In these cases, our result characterises the marginals in the whole regime of temperatures of the model. 
\end{remark}

\begin{remark}[Decoupling and marginals for spherical vectors]
The rapid concentration of $\|g\|$ under $\langle\cdot\rangle$ given by Lemma~\ref{lem:dist_NN} below implies that the components of the random spherical vector $s'$ given by $s_i'=g_i/\|g\|$ are close in distribution, under $\langle \cdot\rangle$, to $g_i/\sqrt N$. Therefore, it is not hard to show that the above results extend when replacing $g_{d_i}$ by $\sqrt N s_{d_i}'$, at least for Lipschitz functions $f$.
\end{remark}

\paragraph{Application to sample means.} Let $f:\R\to\R$ a Lipschitz function with Lipschitz constant $\mbox{Lip}(f)\leq1$. We study the asymptotic behaviour of the sample mean $F(g) := N^{-1} \sum_{i\leq N}f(g_i)$. Without loss of generality we assume that $f(0) = 0$. We also fix $\theta\in T_\rho$.

First observe that
\begin{equation*}
    \Big|\langle F(g)\rangle -  \frac{1}{N}\sum_{d \leq N} \langle f(g_d)\rangle_d\Big| \leq  \frac{1}{N}\sum_{d \leq N}|\langle f(g_d)\rangle - \langle f(g_d)\rangle_d|.
\end{equation*}
Also, define for each $m\geq1$ the function $\bar f_m : \R\to[-m,m]$ such that $\bar f_m(x) = f(x) \mathbb{1}_{|f(x)|\leq m}$. From Theorem \ref{thm:asymptotic_mean} we have, for every $m,d\geq1$,
\begin{equation*}
    \left| \langle \bar f_m(g_d)\rangle - \langle \bar f_m(g_d)\rangle_d \right| \leq K m \,\tilde a_N. 
\end{equation*}
By the fact that $f(0) = 0$ and that $\mbox{Lip}(f)\leq1$ we have that
\begin{equation}\label{eq:L2_bound_m}
    \begin{split}
            \big\langle (f(g_d) -\bar f_m(g_d))^2 \big\rangle & = \big\langle f^2(g_d)\mathbb{1}_{|f(g_d)|\geq m} \big\rangle  \leq \sqrt{\langle g_d^4\rangle \proba(g_d^4 \geq m^4)} \leq \frac{K'}{m^2}.
    \end{split}
\end{equation}
The first inequality used Cauchy-Schwarz and $\mbox{Lip}(f)\leq1$, the second used Markov's inequality and Lemma \ref{lem:bound_gN4} to bound $\langle g_d^4\rangle$ by some $K' > 0$. A similar inequality can be obtained when replacing $\langle \cdot\rangle$ by $\langle \cdot\rangle_d$, which is a normal measure $\mu_{\gamma_d}(\cdot)$ with variance $\sigma^2_{\gamma_d}$ when restricted to functions depending only on spin coordinate $d$. Then, by letting $m \!=\! 1/\sqrt{\tilde a_N}$ we obtain that there exists a fixed $K'' > 0$ such that
\begin{equation}
    \Big|\langle F(g) \rangle - \frac1N \sum_{d\leq N}\int_\R d\mu_{\gamma_d}(x) f(x) \Big| \leq K''\sqrt{\tilde a_N}.
\end{equation}
This along with \cite[Theorem 6]{guionnet2005fourier} and Hypothesis \ref{hyp:eigenvalues} implies that
\begin{equation}\label{eq:conv_emp_mean}
    \Big|\langle F(g) \rangle - \iint d\rho(\gamma) d\mu_\gamma(x) f(x) \Big| \xrightarrow{N\to\infty} 0,
\end{equation}
with $\rho(\cdot)$ the $1$-Wasserstein limit of the eigenvalue measure. 

Our second main result concerns the concentration of a family of functions of $g$ under $\langle\cdot\rangle$. It proves that, for a high-temperature condition more restrictive than the one for the previous results, we have a concentration inequality for their variance.

\begin{theorem}[Concentration inequality]\label{thm:therm_var}
    Suppose $2(9+\sqrt{17})\tilde\gamma|\theta| < 1$. Then there exists some strictly concave Hamiltonian $H'_N:\R^N\to\R$ and $K>0$ such that if we denote by $\langle \cdot \rangle'$ the Gibbs mean associated to $H'_N$ with respect to the standard normal measure on $\R^N$ then, for every $\mathcal{C}^1$ and bounded function $f:\R^N\to\R$,
    \begin{equation*}
        \langle(f(g)-\langle f(g)\rangle)^2\rangle \leq K \|f\|_\infty \big\langle\|\nabla f\|^2\big\rangle'. 
    \end{equation*}
    Furthermore, there exists some $\varepsilon > 0$ such that if $\|g\|^2 \geq (1-\varepsilon)N$ then $H_N(g) = H'_N(g)$ and for $\delta >0$ arbitrarily small we have $\sup_{g\in\R^N}H'_N(g)\leq (\tilde\gamma+\delta)N$.
\end{theorem}
As we will see later on, the proof of this result is short and does not rely on the concentration of the free energy. We believe that this strategy to obtain high-temperature concentration bounds can be easily generalised to many other disordered systems. In particular, it could be used to prove the concentration of overlaps at high-temperature.

\paragraph{Further application to sample means.} Let $2(9+\sqrt{17})\tilde\gamma|\theta| < 1$ and assume that $f:\R\to\R$ is $\mathcal{C}^1$ with derivative bounded by $1$. Again, the sample mean $F(g) := N^{-1} \sum_{i\leq N}f(g_i)$ and for each $m\geq1$ define the function $\bar f_m : \R\to[-m,m]$ such that $\bar f_m(x) := f(x) \mathbb{1}_{|f(x)|\leq m}$. By Theorem \ref{thm:therm_var} we have that there is a $K > 0$ such that for every $i,m\geq1$,
\begin{equation*}
    \Big\langle \Big(\frac1N\sum_{i\leq N} \bar f_m(g_i) - \frac1N\sum_{i\leq N}\langle \bar f_m(g_i)\rangle \Big)^2\Big\rangle \leq \frac{K m}{N}.
\end{equation*}
Take $m=N^{1/3}$. From this inequality and \eqref{eq:L2_bound_m}, we get that there exists $K''>0$ such that %$\Var(F(g)) \leq K''/N^{2/3}$.
\begin{equation*}
    \langle(f(g)-\langle f(g)\rangle)^2\rangle \leq \frac{K''}{N^{2/3}}.
\end{equation*}
We thus obtained a concentration bound for the empirical mean of $f$. This together with \eqref{eq:conv_emp_mean} proves that in the regime $2(9+\sqrt{17})\tilde\gamma|\theta| < 1$ we have
\begin{equation}
    \Big\langle\Big(F(g) - \iint d\rho(\gamma) d\mu_\gamma(x) f(x)\Big)^2\Big\rangle \xrightarrow{N\to\infty} 0.\label{eq:conv_emp_mean_2}
\end{equation}

\section{Technical results}\label{sec:technical}

In this section we present some auxiliary concentration results that will be used during the proofs of the theorems provided in the previous section.

\begin{lemma}[Marginal and concentration of the norm]\label{lem:dist_NN}
    $(i)$ The marginal of $\|g\|^2$ with respect to the measure induced by $\langle\cdot\rangle$ is given by the $\chi^2$ distribution with $N$ degrees of freedom. $(ii)$ There exists a constant $K > 0$ such that
    \begin{equation*}
        \Big\langle\Big( \frac{N}{\|g\|^2} - 1 \Big)^2 \Big\rangle \leq \frac{K}{N}.
    \end{equation*}
%\end{corollary}
\end{lemma}
\begin{proof}
    $(i)$ follows from the fact that Hamiltonian $H_N(g)$ does not depend on $\|g\|$ and thus simplifies when evaluating expressions of the form $\langle f(\|g\|^2) \rangle$, and the remaining measure on $g$ is gaussian. From $(i)$, under $\langle \cdot\rangle$ the variable $N/\|g\|^2$ has a scaled-inverse-$\chi^2$ distribution with $N$ degrees of freedom. $(ii)$ then holds because this distribution has mean $N/(N-2)$ (for $N > 2$) and variance equal to $2N^2/((N-2)^2(N-4))$ (for $N > 4)$.
\end{proof}

\subsection{Concentration of the energy}

We now obtain concentration bounds for the intensive energy $h_N := N^{-1} H_N$. These are based on the asymptotic formula for the free energy proved in \cite[Theorem 6]{guionnet2005fourier}. 
\begin{lemma}[Concentration of  energy]\label{lem:en_contr}
    If $\theta\in T_\rho$, then for every $\varepsilon > 0$ there exist $K,K' > 0$ s.t. $\proba(|h_N-v(\theta)| > \varepsilon) \leq K\exp(-K'\varepsilon^2 N)$. Furthermore, $\lim_{N\to\infty} \langle(h_N-v(\theta))^2\rangle = 0$.
\end{lemma}
\begin{proof}
    By \cite[Theorem 6]{guionnet2005fourier} we know that, under Hypothesis \ref{hyp:eigenvalues}, for every $\theta\in T_\rho$
    \begin{equation}\label{eq:guio_limit}
        \lim_{N\to\infty} f_N =\lim_{N\to\infty}\frac1{N} \ln Z_N  = \frac{1}{2}\int_0^{2\theta} R_\rho(x)dx.
    \end{equation}
    We will first see that this limit implies that $b_N:=|\langle h_N\rangle_\theta -v(\theta)|\to0$ as $N\to\infty$. First, if $G,g:\R\to\R$ are convex, we have for any $\delta > 0$ (see, e.g., \cite[Lemma 3.2]{panchenko2013sherrington})
    \begin{equation}\label{eq:conv_bound}
            |G'(\theta) - g'(\theta)| \leq (g'(\theta+\delta) - g'(\theta)) + (g'(\theta) - g'(\theta-\delta)) + \delta^{-1} \sum_{y \in \mathcal{Y}} |G(y)-g(y)|
    \end{equation}
    where $\mathcal{Y}:=\{\theta-\delta, \theta, \theta+\delta\}$. Because $R_\rho(x)$ is strictly increasing, $\frac12 \int_0^{2\theta} R_\rho(x) dx$ is convex in $\theta$. The finite-size free energy $f_N(\theta)$ is convex too, its second derivative being proportional to the energy variance. Thus, we have for every $\delta > 0$ verifying $(\theta-\delta,\theta+\delta)\in T_\rho$,
    \begin{equation*}
        |\langle h_N\rangle_\theta - v(\theta)| \leq v(\theta+\delta) - v(\theta-\delta) + \delta^{-1} \sum_{y\in\{-\delta,0,\delta\}} \Big|f_N(\theta+y) - \frac{1}{2}\int_0^{2(\theta+y)}R_\rho(x)dx\Big|.
    \end{equation*}
    Taking the lim sup over $N\to\infty$ and using \eqref{eq:guio_limit} we get that $\lim\sup_{N\to\infty}b_N$ is upper bounded by $v(\theta+\delta) - v(\theta-\delta)$. And by the differentiability of $R_\rho$ at $2\theta$, we have that the right hand side goes to $0$ as $\delta\to0$. This implies that $\lim_{N\to\infty}b_N(\theta) = 0$ for any $\theta\in T_\rho$.
    
    Now, observe that for all $t > 0$ by the mean value theorem and the continuity of $f_N$,
    \begin{equation*}
        \left\langle \exp(t N h_N)\right\rangle_\theta = \frac{Z_N(\theta+t)}{Z_N(\theta)} = \exp(N f_N(\theta+t) - N f_N(\theta)) = \exp(t N \langle h_N \rangle_{\xi+\theta}),
    \end{equation*}
    for some $\xi\in[0,t]$. Because $\langle h_N \rangle_{\theta}$ is increasing in $\theta$ (its derivative is a variance) we have
    \begin{equation*}
        \begin{split}
            \left\langle \exp(t N (h_N-v(\theta)))\right\rangle_\theta & = \exp(t N (\langle h_N \rangle_{\xi+\theta}-v(\theta)))\leq \exp(t N (\langle h_N \rangle_{t+\theta}-v(\theta)))\,.
        \end{split}
    \end{equation*}
    Therefore, by the fact that $R_\rho$ has a continuous derivative near $2\theta$ we have
    \begin{equation*}
        \left\langle \exp(t N (h_N-v(\theta)))\right\rangle_\theta \leq \exp(t^2 N K + t N |\langle h_N\rangle_{t+\theta}-v(t+\theta)|),
    \end{equation*}
    for all $t > 0$ sufficiently small. Let $\varepsilon > 0$ be sufficiently small and fix $t = \varepsilon/4K$. Then, $\proba(h_N - v(\theta) > \varepsilon) \leq \exp(N\varepsilon(-\varepsilon/4+(b_N-\varepsilon/2))/(4K))$. Notice that because $b_N\to0$, for every $\varepsilon >0$ there is some $K'>0$ such that $\exp(N\varepsilon(b_N -\varepsilon/2)/(4K))\leq K'$. Thus,
    \begin{equation*}
        \proba(h_N - v(\theta) > \varepsilon) \leq K' \exp(-\varepsilon^2 N/(16 K)).
    \end{equation*}
    An almost identical argument yields $\proba(v(\theta) - h_N > \varepsilon) \leq K'\exp(-\varepsilon^2 N/(16 K))$. From which we conclude the first part of the lemma. 
    
    To get the second part, it is enough to use that for every $\varepsilon > 0$ we have
    \begin{equation*}
        \big\langle(h_N-v(\theta))^2\big\rangle \leq \varepsilon^2 + (\tilde\gamma^2+v(\theta)^2)K'\exp(-\varepsilon^2 N/(16 K)).
    \end{equation*}
    Which implies that $\lim\sup_{N\to\infty} \big\langle(h_N-v(\theta))^2\big\rangle \leq \varepsilon^2$. The conclusion then follow because $\varepsilon$ is arbitrarily small.
\end{proof}

\subsection{Uniform bound for the moments}

Here we prove uniform bounds for the moments of the coordinates of $g$. For this, we will need the following equivalent condition for the high-temperature regime considered.

\begin{lemma}[Equivalent high-temperature condition]\label{lem:equiv_cond}
    $\theta\in T_\rho \Leftrightarrow 2|\theta|(\tilde\gamma-v(\theta))<1$.
\end{lemma}
\begin{proof}
    Assume $\theta > 0$. Because of the monotonicity of $K_\rho$, the condition $2\theta < S_{max}$ holds iff $K_\rho(2\theta) < \tilde \gamma$. This in turn is equivalent to $R_\rho(2\theta) < \tilde\gamma - 1/(2\theta)$. By the definition of $v(\theta)$ this is the same as $2\theta(\tilde\gamma-v(\theta))<1$. If $\theta < 0$, the conclusion follows in a similar way.
\end{proof}

\begin{lemma}[Boundedness of moments]\label{lem:bound_gN4}
    Let $d\geq1$ and assume that $N\geq d$. If $\theta\in T_\rho$ then for every $n\geq1$ there exists $K > 0$ such that $\langle g_d^{2n}\rangle  \leq K$ uniformly in $N$.
\end{lemma}
\begin{proof}
    We will prove this by induction. First notice that by gaussian integration by parts we have that for every $n\geq1$
    \begin{equation}\label{eq:int_by_parts}
        \langle g_d^{2n} \rangle = (2n-1) \langle g_d^{2(n-1)} \rangle + 2 \theta N \Big\langle (\gamma_d-h_N)\frac{g_d^{2n}}{\|g\|^2} \Big\rangle.
    \end{equation}
    Also, observe that for all $n\geq1$ we have that there is some $K_n > 0$ such that
    \begin{equation}\label{eq:rough_bound_normg}
        \langle g_d^{2n} \rangle \leq \langle \|g\|^{2n}\rangle \leq K_n N^n,
    \end{equation}
    where for the last inequality we used Lemma \ref{lem:dist_NN}. Finally, because $\theta \in T_\rho$, by Lemma \ref{lem:equiv_cond} there is some $\varepsilon > 0$ small enough so that
    \begin{equation*}
        1-\frac{2\theta(\tilde\gamma-v+\varepsilon)}{1-\varepsilon} > 0.
    \end{equation*}
    For the rest of the proof, we will regard $\varepsilon > 0$ to be fixed in this way.
    
    To start the induction, consider equation \eqref{eq:int_by_parts} for $n=1$:
    \begin{equation}\label{eq:bound_first_moment}
        \langle g_d^2\rangle = 1 + 2 \theta N \Big\langle (\gamma_d-h_N)\frac{g_d^{2}}{\|g\|^2} \Big\rangle.
    \end{equation}
    By Laurent-Massart's bound \cite[Lemma 1]{laurent2000adaptive}, $\proba(\|g\|^2 \leq (1-\varepsilon)N)\le \exp(-\varepsilon^2 N/4)$. Define the event $A_\varepsilon:=\{\|g\|^2 \leq (1-\varepsilon)N\}\cup\{|h_N-v|\geq\varepsilon\}$. Then, by Lemma \ref{lem:en_contr}, equation \eqref{eq:rough_bound_normg}, Cauchy-Schwarz, and Laurent-Massart's bound, there are  $K',K''>0$ such that 
    \begin{equation}\label{eq:bound_aux1}
        \begin{split}
            %\Big\langle N(\gamma_d-h_N)\frac{g_d^{2}}{\|g\|^2}\Big\rangle & = 
            \Big\langle N(\gamma_d-h_N)\frac{g_d^{2}}{\|g\|^2}(\mathbb{1}_{A_\varepsilon^c}+\mathbb{1}_{A_\varepsilon})\Big\rangle \leq \frac{\tilde\gamma-v+\varepsilon}{1-\varepsilon}\langle g_d^{2}\rangle  + K' N \exp({-K''\varepsilon^2 N}).
        \end{split}
    \end{equation}
    And equations \eqref{eq:bound_first_moment} and \eqref{eq:bound_aux1} together with the value chosen for $\varepsilon$ imply that $\langle g_d^2\rangle$ is uniformly bounded in $N$.
    
    To advance the induction, we assume that there exists $K'''>0$ such that $\langle g_d^{2(n-1)}\rangle \leq K'''$ uniformly in $N$. In the same way we got \eqref{eq:bound_aux1}, there exists $K^{(4)},K^{(5)} > 0$ such that
    \begin{equation}\label{eq:bound_aux2}
        \Big\langle N(\gamma_d-h_N)\frac{g_d^{2n}}{\|g\|^2}\Big\rangle \leq \frac{\tilde\gamma-v+\varepsilon}{1-\varepsilon}\langle g_d^{2n}\rangle  + K^{(4)} N^n \exp({-K^{(5)}\varepsilon^2 N}).
    \end{equation}
    Putting together equations \eqref{eq:int_by_parts} and \eqref{eq:bound_aux2} as well as the induction hypothesis it follows that for some fixed $K^{(6)}>0$,
    \begin{equation*}
        \Big(1-\frac{2\theta(\tilde\gamma-v+\varepsilon)}{1-\varepsilon}\Big)\langle g_d^{2n} \rangle \leq (2n-1)K''' + K^{(6)}.
    \end{equation*}
    This ends the induction and the proof.
\end{proof}

\subsection{Approximation of the Gibbs mean}

Here we prove a simple general approximation result for Gibbs measures. In a nutshell, it shows that perturbations of the Hamiltonian that involve functions with ``small means'' do not affect the asymptotic values of means of bounded functions. 

\begin{proposition}[Perturbation control]\label{prop:pert_cont}
    Let $\varphi_N, \epsilon_N, \epsilon'_N : \R^N \to \R$ be three random functions such that $\epsilon_N$ and $\epsilon'_N$ are almost surely non-negative and, for all $t_1,t_2\in[0,1]$, the function $\exp(\varphi_N(X)+t_1\epsilon_N(X)-t_2\epsilon'_N(X))$ is almost surely integrable with respect to $X$. Let $\langle\cdot\rangle_{t_1,t_2}$ be the mean of the Gibbs measure associated with Hamiltonian $\varphi_N(X)+t_1\epsilon_N(X)-t_2\epsilon'_N(X)$. Then, for every bounded function $f:\R^N\to \R$ we have that
    \begin{equation*}
        \left| \langle f(X) \rangle_{1,1} - \langle f(X) \rangle_{0,0} \right| \leq 2 \|f\|_\infty \left( \langle \epsilon_N(X) \rangle_{1,1} + \langle \epsilon'_N(X) \rangle_{0,0} \right).
    \end{equation*}
\end{proposition}
\begin{proof}
    For this, note that
    \begin{equation*}
        \frac{d}{dt_1} \langle \epsilon_N \rangle_{t_1,1} = \langle \epsilon_N^2 \rangle_{t_1,1} - \langle \epsilon_N \rangle_{t_1,1}^2 \geq 0. 
    \end{equation*}
    We then have that for all $t_1\in[0,1]$, $\langle \epsilon_N \rangle_{t_1,1} \leq \langle \epsilon_N \rangle_{1,1}$. Thus,
    \begin{equation*}
        \begin{split}
            \frac{d}{dt_1} \langle f \rangle_{t_1,1}  = \left\langle f \left(\epsilon_N - \langle \epsilon_N  \rangle_{t_1,1}\right)  \right\rangle_{t_1,1} & \leq \|f\|_\infty \left\langle \left|\epsilon_N - \langle \epsilon_N  \rangle_{t_1,1}\right|  \right\rangle_{t_1,1} \\
            & \leq 2 \|f\|_\infty \langle \epsilon_N \rangle_{t_1,1} \leq2 \|f\|_\infty \langle \epsilon_N \rangle_{1,1},
        \end{split}
    \end{equation*}
    which proves that $|\langle f \rangle_{1,1} - \langle f \rangle_{0,1}| \leq 2 \|f\|_\infty \langle \epsilon_N \rangle_{1,1}$.
    
    In a similar way, we have that
    \begin{equation*}
        \frac{d}{dt_2} \langle \epsilon_N \rangle_{0,t_2} = -\langle (\epsilon'_N)^2 \rangle_{0,t_2} + \langle \epsilon'_N \rangle_{0,t_2}^2 \leq 0. 
    \end{equation*}
    We then get that for all $t_2\in[0,1]$, $\langle \epsilon'_N \rangle_{0,t_2} \leq \langle \epsilon'_N \rangle_{0,0}$. Thus,
    \begin{equation*}
        \begin{split}
            \frac{d}{dt_2} \langle f \rangle_{0,t_2}  = \left\langle f \left(\langle \epsilon'_N  \rangle_{0,t_2} - \epsilon'_N\right)  \right\rangle_{0,t_2} & \leq \|f\|_\infty \left\langle \left|\epsilon'_N - \langle \epsilon'_N  \rangle_{0,t_2}\right|  \right\rangle_{0,t_2}\leq 2 \|f\|_\infty \langle \epsilon'_N \rangle_{0,0},
        \end{split}
    \end{equation*}
    which proves that $|\langle f \rangle_{0,1} - \langle f \rangle_{0,0}| \leq 2 \|f\|_\infty \langle \epsilon'_N \rangle_{0,0}$. This concludes the proof.
\end{proof}

If the small perturbation is a generic function $g$, $\epsilon_N$ and $\epsilon'_N$ can be taken to be equal to its positive $g^+(X):=g(X)\mathbb{1}_{g(x)\geq 0}$ and negative $g^-(X):=|g(X)|\mathbb{1}_{g(X)< 0}$ parts. Then,
\begin{equation}
    | \langle f \rangle_{1,1} - \langle f \rangle_{0,0} | \leq 2 \|f\|_\infty \big( \langle g^+ \rangle_{1,1} + \langle g^- \rangle_{0,0} \big) \leq 2 \|f\|_\infty \big( \langle |g| \rangle_{1,1} + \langle |g| \rangle_{0,0} \big).\label{ConcId}
\end{equation}

\section{Proofs of the main results}\label{sec:proofs}

\subsection{Proof of Theorem \ref{thm:asymptotic_mean} using the cavity method}

The proof of this theorem will be based on a cavity argument and the approximation result given by Proposition \ref{prop:pert_cont}. As in Section \ref{sec:main}, let $d\geq1$ and (for $N\geq d$) $\bar g = (g_1,\dots,g_{d-1},g_{d+1},\dots,g_N)\in\R^{N-1}$. Observe that for all $g\in\R^N$,
\begin{equation}\label{eq:change_denom}
    \frac{1}{\|g\|^2} = \frac{1}{\|\bar g\|^2} - \frac{g_d^2}{\|g\|^2 \|\bar g\|^2}.
\end{equation}
From this we have that
\begin{equation*}
    H_N(g) = NV_N + \frac{N}{\|g\|^2} ( \gamma_d - V_N ) g_d^2 \quad \mbox{with}\quad V_N := \sum_{i\leq N, \ i\neq d} \gamma_i \frac{g_i^2}{\|\bar g\|^2}.
\end{equation*}
Recall definition \eqref{eq:def_Hd} and let
\begin{equation}\label{eq:error_ham}
    \epsilon_N(g) := \left(\frac{N}{\|g\|^2}-1\right) (\gamma_d-V_N)  g_d^2 + (v-V_N)  g_d^2,
\end{equation}
with, as in Section \ref{sec:setting}, $v:=R_\rho(2\theta)$. We then have that
$H_N(g) = H_d(g) + \epsilon_N(g)$. Notice that if we remove the term $\epsilon_N$ from the Hamiltonian, the resulting Gibbs measure has the $d$-th coordinate decoupled from the rest. As in Section \ref{sec:main}, we will denote by $\langle \cdot \rangle_d$ the Gibbs mean of the measure defined by the Hamiltonian $H_d(g)$. The idea will then be to use Proposition \ref{prop:pert_cont} to connect the mean values of the original measure with the ones of the decoupled measure. For this we will need concentration bounds for $N/\|g\|^2$ and $V_N$ on both the original and decoupled measures.

\begin{lemma}[Norm concentration]\label{lem:conc_NNd}
    There exists some constant $K > 0$ such that
    \begin{equation*}
        \Big\langle\Big( \frac{N}{\|g\|^2} - 1 \Big)^2 \Big\rangle_d \leq \frac{K}{N}.
    \end{equation*}
\end{lemma}
\begin{proof}
    We will first use that by \eqref{eq:change_denom} and the triangular inequality,
    \begin{equation}\label{eq:triang_ineq}
        \sqrt{\Big\langle\Big( \frac{N}{\|g\|^2} - 1 \Big)^2 \Big\rangle_d} \leq \sqrt{\Big\langle\Big( \frac{N}{\|\bar g\|^2} - 1 \Big)^2 \Big\rangle_d} + \sqrt{\Big\langle \frac{N^2 g_d^4}{\|g\|^4 \|\bar g\|^4} \Big\rangle_d}.
    \end{equation}
    The first term is $O(1/\sqrt{N})$ by Lemma \ref{lem:dist_NN} and the fact that $\langle f(\|\bar g\|)\rangle_d=\langle f(\|\bar g\|)\rangle$. For the second term of \eqref{eq:triang_ineq}, we have
    \begin{equation*}
        \Big\langle \frac{N^2 g_d^4}{\|g\|^4 \|\bar g\|^4} \Big\rangle_d \leq \Big\langle \frac{N^2 g_d^4}{\|\bar g\|^8} \Big\rangle_d = \frac{1}{N^2} \Big\langle \frac{N^4}{\|\bar g\|^8} \Big\rangle_d \langle g_d^4 \rangle_d.
    \end{equation*}
    And $\langle N^4/\|\bar g\|^8 \rangle_d = 1 + o_N(1)$ as it is the fourth moment of a scaled-inverse-$\chi^2$ random variable. Under the measure $\langle \cdot\rangle_d$, $g_d$ is distributed as a centred gaussian random variable of variance $\sigma_{\gamma_d}^2 = 1/(1 + 2 \theta (v - \gamma_d))$. Because $\theta\in T_\rho$, by Lemma \ref{lem:equiv_cond} we have that $1 + 2 \theta v > 2\theta\tilde \gamma$ and thus $\sigma_{\gamma_d}^2$ is uniformly bounded with respect to $N$ and $\gamma_d$. This means that $\langle g_d^4\rangle_d$ is uniformly bounded and thus the last term in \eqref{eq:triang_ineq} is $O(1/N)$ too. 
\end{proof}
\begin{lemma}[Concentration for decoupled model]\label{lem:conc_eng_d} 
    $\theta\in T_\rho\Rightarrow  \lim_{N\to\infty}\langle(V_N-v(\theta))^2\rangle_d = 0$.
\end{lemma}
\begin{proof}
    For $u\in[0,1]$ define $\langle\cdot\rangle_u$ as the Gibbs measure associated to the Hamiltonian $(N-(1-u))\theta V_N$ with respect to the standard normal measure in $\R^{N-1}$. At $u=0$ this measure coincides with the original measure $\langle\cdot\rangle$ for a system of size $N-1$ while at $u=1$ it corresponds to the marginal of $\langle\cdot\rangle_d$ for the $N-1$ coordinates different from $d$. Observe that $d\langle (V_N-\langle V_N\rangle_u)^2\rangle_u/du \leq 3\tilde\gamma\theta \langle (V_N-\langle V_N\rangle_u)^2\rangle_u$. Then by Gronwall's inequality we have that $\langle (V_N-\langle V_N\rangle_d)^2\rangle_d \leq \exp({3\tilde\gamma\theta}) \langle (V_N-\langle V_N\rangle)^2\rangle$, and the right-hand side goes to $0$ by Lemma \ref{lem:en_contr}. Finally, $\lim_{N\to\infty}\langle V_N\rangle_d = v(\theta)$ because the limit of the free energy is the same for every value of $u\in[0,1]$ and when $u=0$, as we saw before, $\lim_{N\to\infty}\langle V_N\rangle = v(\theta)$.
\end{proof}

The intensive energy $h_N:=N^{-1}H_N$ and $V_N$ are upper bounded by $\tilde\gamma$. It is then easy to see that $\langle |h_N - V_N|\rangle \leq K'/N$ for some $K' >0$. Then, by Lemmas \ref{lem:dist_NN}, \ref{lem:bound_gN4}, and \ref{lem:conc_NNd} we have that when $\theta\in T_\rho$ and is finite there exists some finite $K > 0$ such that
\begin{equation}
    \langle |\epsilon_N|\rangle \leq K\Big(\frac{1}{\sqrt N}+\sqrt{a_N}\Big) \,\,\,\,\, \mbox{ and } \,\,\,\,\, \langle |\epsilon_N|\rangle_d \leq K\Big(\frac{1}{\sqrt N}+\sqrt {a_N}\Big).
\end{equation}
Then, by Proposition \ref{prop:pert_cont}, in particular \eqref{ConcId}, the Theorem \ref{thm:asymptotic_mean} is proved.

\subsection{Proof of Theorem \ref{thm:therm_var} by convex extension and Brascamp-Lieb's inequality}

The proof of this theorem is based on constructing a concave approximation of $H_N$ that coincides with it in a set of high probability. Then, we use Brascamp-Lieb's inequality for log-concave measures for the Gibbs measure associated to this approximation.

\begin{lemma}\label{lem:conv_ext}
    Suppose $2(9+\sqrt{17})\tilde\gamma|\theta| < 1$. Then, there exists some Hamiltonian $H'_N:\R^N\to\R$ with Hessian upper bounded in the Loewner order by $b\mathbb{I}$, $0< b < 1/2$, such that, if we denote by $\langle \cdot \rangle'$ the Gibbs measure associated to $H'_N$ w.r.t. the standard normal measure, there are constants $\alpha,K > 0$ such that for every bounded function $f:\R^N\to\R$
    \begin{equation*}
        \left|\langle f\rangle - \langle f\rangle'\right| \leq \|f\|_\infty K \exp({-\alpha N}),
    \end{equation*}
    with $K$ and $\alpha$ not depending on the function $f$. Furthermore, there exists some $\varepsilon > 0$ such that if $\|g\|^2 \geq (1-\varepsilon)N$ then $H_N(g) = H'_N(g)$ and for all $\delta >0$ $\sup_{g\in\R^N}H'_N(g)\leq (\tilde\gamma+\delta)N$.
\end{lemma}
\begin{proof}
    First, notice that the Hessian $\mathcal{H}\in\R^{N\times N}$ of $H_N(g)$ (which exists for every $g\neq0$) has, for $i,j\in[N]$, components given by
    \begin{equation*}
        \mathcal{H}_{ij} = -4  (\gamma_i+\gamma_j-2h_N) \frac{N g_i g_j}{\|g\|^4} - \Big(2 (h_N-\gamma_i) \frac{N}{\|g\|^2} \Big) \delta_{ij}.
    \end{equation*}
    By naive bounds, we get that for every $v\in S^{N-1}(1)$
    \begin{equation}\label{eq:hessian_bound}
        v^\intercal \mathcal{H} v \leq \frac{16  \tilde\gamma N}{\|g\|^2}. 
    \end{equation}
    It is easy to check that because $2(9+\sqrt{17})\tilde\gamma|\theta| < 1$, there is some $\varepsilon > 0$ such that $2\sqrt{|\theta| \tilde \gamma} < \varepsilon < 1 - 16 |\theta| \tilde \gamma$. Choose some $\varepsilon >0$ in this way and let $\delta > 0$ be a small constant to be fixed later on. Define the events 
    \begin{align*}
        A_\varepsilon &:= \left\{g\in\R^N : (1-\varepsilon) N \leq \|g\|^2 \leq N\right\},\\
        B_{\varepsilon,\delta} &:= \left\{g\in\R^N : (1-\varepsilon-\delta) N < \|g\|^2 < (1 + \delta) N\right\}.
    \end{align*}
    Notice that if $\|g\|^2 \geq (1-\varepsilon-\delta)N$, by equation \eqref{eq:hessian_bound}, the norm of the Hessian is upper bounded by $a := 16|\theta| \tilde\gamma/(1-\varepsilon-\delta)$. Then, by the choice of $\varepsilon$, the function $\theta H_N-a\|g\|^2/2$ is $\mathcal{C}^2$ and has negative definite Hessian in $A_\varepsilon$ and $B_{\varepsilon,\delta}$ for all $\delta > 0$ small enough. Let $k:B_{\varepsilon,\delta} \to\R$ be the restriction of $\theta H_N-a\|g\|^2/2$ to $B_{\varepsilon,\delta}$. Because $A_\varepsilon$ is a compact set with open neighbourhood $B_{\varepsilon,\delta}$, from \cite[Theorem 3.2]{yan2014extension} and the details of its proof we can assess that there is a concave extension $\hat k$ of $k$ to the whole ball $B(0,\sqrt{N})$ that is $\mathcal{C}^2$, has negative definite Hessian, and that verifies
    \begin{equation}
        \sup_{g\in B(0,\sqrt{N})} \hat k(g) \leq \sup_{g\in B_{\varepsilon,\delta}} k(g) \leq -7 \tilde\gamma|\theta|N.
    \end{equation}
    This means that if we let $C_\varepsilon := \{g\in\R^N :  \|g\|^2 \leq (1-\varepsilon) N\}$ we can define a new Hamiltonian $H'_N(g) := \theta^{-1}(\hat k(g)+a\|g\|^2/2) \mathbb{1}_{C_\varepsilon} + H_N(g) \mathbb{1}_{C^c_\varepsilon}$
    that coincides with $H_N$ on $C^c_\varepsilon$, has a Hessian upper bounded in the Loewner order by $a\mathbb{I}$ on $\R^N$, and $\sup_{g\in C_\varepsilon} H'_N(g)$ is smaller or equal than $\tilde \gamma( 1 + \frac{8\delta}{1-\varepsilon-\delta}) N$. Fix $\delta > 0$ small enough so that 
    \begin{equation}\label{eq:fix_delta}
        |\theta| \tilde \gamma\Big( 1 + \frac{8\delta}{1-\varepsilon-\delta}\Big) < \frac{\varepsilon^2}{4}.
    \end{equation}
    
    Let us call $\langle \cdot \rangle'$ the Gibbs measure defined by the Hamiltonian $H'_N$ with respect to the standard normal measure on $\R^N$ and $Z'_N$ its partition function. We will now prove that for every bounded function $f(g)$ we have that $|\langle f\rangle - \langle f\rangle'|$ goes to $0$ as $N\to+\infty$. To show this, first note that for every $A,B,A',B'\in\R$ such that there exists some $C_1,C_2>0$ with $|A'/B'|\leq C_1$ and $1/B\leq C_2$, it holds that
    \begin{equation*}
        \left|\frac{A}{B}-\frac{A'}{B'}\right| \leq C_2|A-A'| + C_1C_2|B-B'|.
    \end{equation*}
    Because $H_N(g) \geq 0$ for all $g\in\R^N$, we have that $Z_N\geq1$ and because $f(g)$ is bounded we have $\langle f \rangle' \leq \|f\|_\infty$. Then, to prove that $|\langle f\rangle - \langle f\rangle'|$ vanishes it is enough to see that
    \begin{equation*}
        \int dG(g) \big|\exp \theta H_N(g) - \exp \theta H'_N(g)\big|\xrightarrow{N\to\infty} 0.
    \end{equation*}
    And because both Hamiltonians only differ on $C_\varepsilon$ and are bounded on it, we have that
    \begin{equation*}
        \int dG \big|e^{\theta H_N} - e^{\theta H'_N}\big| \leq (e^{\theta\tilde\gamma N}+e^{|\theta| \tilde \gamma + \frac{8\delta|\theta| \tilde \gamma}{1-\varepsilon-\delta}}) \int_{C_\varepsilon}dG \leq (e^{\theta\tilde\gamma N}+e^{|\theta| \tilde \gamma + \frac{8\delta|\theta| \tilde \gamma}{1-\varepsilon-\delta}}) e^{-\frac{\varepsilon^2}{4}N},
    \end{equation*}
    where we used Laurent-Massart's bound \cite[Lemma 1]{laurent2000adaptive} for the inequality. By the choice of $\varepsilon$ and $\delta$,  there is some $\alpha > 0$ such that the right-hand side is equal to $\exp(-\alpha N)$. 
\end{proof}

Because the norm of the Hessian of the Hamiltonian $H_N'$ is strictly smaller than $1/2$ then the resulting Gibbs measure is log-concave. The conclusion of Theorem \ref{thm:therm_var} then follows directly from Lemma \ref{lem:conv_ext} and Brascamp-Lieb's inequality \cite[Theorem 4.1]{brascamp1976extensions}.


\begin{thebibliography}{10}

    \bibitem{aizenman2003extended}
    Michael Aizenman, Robert Sims, and Shannon~L Starr.
    \newblock Extended variational principle for the sherrington-kirkpatrick
      spin-glass model.
    \newblock {\em Physical Review B}, 68(21):214403, 2003.
    
    \bibitem{jean2019adaptive}
    Jean Barbier and Nicolas Macris.
    \newblock The adaptive interpolation method: a simple scheme to prove replica
      formulas in bayesian inference.
    \newblock {\em Probability theory and related fields}, 174(3):1133--1185, 2019.
    
    \bibitem{jean2019adaptive2}
    Jean Barbier and Nicolas Macris.
    \newblock The adaptive interpolation method for proving replica formulas.
      applications to the curie--weiss and wigner spike models.
    \newblock {\em Journal of Physics A: Mathematical and Theoretical},
      52(29):294002, 2019.
    
    \bibitem{belius2019tap}
    David Belius and Nicola Kistler.
    \newblock The tap--plefka variational principle for the spherical sk model.
    \newblock {\em Communications in Mathematical Physics}, 367(3):991--1017, 2019.
    
    \bibitem{arous2001aging}
    G\'erard Ben~Arous, Amir Dembo, and Alice Guionnet.
    \newblock Aging of spherical spin glasses.
    \newblock {\em Probability theory and related fields}, 120(1):1--67, 2001.
    
    \bibitem{bhattacharya2016high}
    Bhaswar~B Bhattacharya and Subhabrata Sen.
    \newblock High temperature asymptotics of orthogonal mean-field spin glasses.
    \newblock {\em Journal of Statistical Physics}, 162(1):63--80, 2016.
    
    \bibitem{brascamp1976extensions}
    Herm~Jan Brascamp and Elliott~H Lieb.
    \newblock On extensions of the brunn-minkowski and pr{\'e}kopa-leindler
      theorems, including inequalities for log concave functions, and with an
      application to the diffusion equation.
    \newblock {\em Journal of Functional Analysis}, 22(4):366--389, 1976.
    
    \bibitem{chen2013aizenman}
    Wei-Kuo Chen.
    \newblock The aizenman-sims-starr scheme and parisi formula for mixed $ p
      $-spin spherical models.
    \newblock {\em Electronic Journal of Probability}, 18:1--14, 2013.
    
    \bibitem{coja2018information}
    Amin Coja-Oghlan, Florent Krzakala, Will Perkins, and Lenka Zdeborov{\'a}.
    \newblock Information-theoretic thresholds from the cavity method.
    \newblock {\em Advances in Mathematics}, 333:694--795, 2018.
    
    \bibitem{crisanti1993sphericalp}
    Andrea Crisanti, Heinz Horner, and Hans-Jurgen Sommers.
    \newblock The sphericalp-spin interaction spin-glass model.
    \newblock {\em Zeitschrift f{\"u}r Physik B Condensed Matter}, 92(2):257--271,
      1993.
    
    \bibitem{crisanti1992spherical}
    Andrea Crisanti and Hans-Jurgen Sommers.
    \newblock The spherical p-spin interaction spin glass model: the statics.
    \newblock {\em Zeitschrift f{\"u}r Physik B Condensed Matter}, 87(3):341--354,
      1992.
    
    \bibitem{donoho2013information}
    David~L Donoho, Adel Javanmard, and Andrea Montanari.
    \newblock Information-theoretically optimal compressed sensing via spatial
      coupling and approximate message passing.
    \newblock {\em IEEE transactions on information theory}, 59(11):7434--7464,
      2013.
    
    \bibitem{fan2020AMProt}
    Zhou Fan.
    \newblock Approximate message passing algorithms for rotationally invariant
      matrices.
    \newblock {\em arXiv preprint arXiv:2008.11892}, 2020.
    
    \bibitem{fan2021ortho}
    Zhou Fan and Yihong Wu.
    \newblock The replica-symmetric free energy for ising spin glasses with
      orthogonally invariant couplings.
    \newblock {\em arXiv preprint arXiv:2105.02797}, 2021.
    
    \bibitem{kurchan2021annealed}
    Laura Foini and Jorge Kurchan.
    \newblock Annealed averages in spin and matrix models.
    \newblock {\em arXiv preprint arXiv:2104.04363}, 2021.
    
    \bibitem{gabrie2019entropy}
    Marylou Gabri{\'e}, Andre Manoel, Cl{\'e}ment Luneau, Jean Barbier, Nicolas
      Macris, Florent Krzakala, and Lenka Zdeborov{\'a}.
    \newblock Entropy and mutual information in models of deep neural networks.
    \newblock {\em Journal of Statistical Mechanics: Theory and Experiment},
      2019(12):124014, 2019.
    
    \bibitem{gerbelot2020asymptotic}
    Cedric Gerbelot, Alia Abbara, and Florent Krzakala.
    \newblock Asymptotic errors for teacher-student convex generalized linear
      models (or: How to prove kabashima's replica formula).
    \newblock {\em arXiv preprint arXiv:2006.06581}, 2020.
    
    \bibitem{guerra2002thermodynamic}
    Francesco Guerra and Fabio~Lucio Toninelli.
    \newblock The thermodynamic limit in mean field spin glass models.
    \newblock {\em Communications in Mathematical Physics}, 230(1):71--79, 2002.
    
    \bibitem{guionnet2005fourier}
    Alice Guionnet and Mylene Ma{\i}da.
    \newblock A fourier view on the r-transform and related asymptotics of
      spherical integrals.
    \newblock {\em Journal of functional analysis}, 222(2):435--490, 2005.
    
    \bibitem{kabashima2003cdma}
    Yoshiyuki Kabashima.
    \newblock A cdma multiuser detection algorithm on the basis of belief
      propagation.
    \newblock {\em Journal of Physics A: Mathematical and General}, 36(43):11111,
      2003.
    
    \bibitem{kosterlitz1976spherical}
    John~M Kosterlitz, David~J Thouless, and Raymund~C Jones.
    \newblock Spherical model of a spin-glass.
    \newblock {\em Physical Review Letters}, 36(20):1217, 1976.
    
    \bibitem{laurent2000adaptive}
    Beatrice Laurent and Pascal Massart.
    \newblock Adaptive estimation of a quadratic functional by model selection.
    \newblock {\em Annals of Statistics}, pages 1302--1338, 2000.
    
    \bibitem{ma2021analysis}
    Junjie Ma, Ji~Xu, and Arian Maleki.
    \newblock Impact of the sensing spectrum on signal recovery in generalized
      linear models.
    \newblock {\em arXiv preprint arXiv:2111.03237}, 2021.
    
    \bibitem{maillard2019hightemp}
    Antoine Maillard, Laura Foini, Alejandro~Lage Castellanos, Florent Krzakala,
      Marc M{\'e}zard, and Lenka Zdeborov{\'a}.
    \newblock High-temperature expansions and message passing algorithms.
    \newblock {\em Journal of Statistical Mechanics: Theory and Experiment},
      2019(11):113301, 2019.
    
    \bibitem{mezard2009information}
    Marc M\'ezard and Andrea Montanari.
    \newblock {\em Information, physics, and computation}.
    \newblock Oxford University Press, 2009.
    
    \bibitem{mezard1987spin}
    Marc M{\'e}zard, Giorgio Parisi, and Miguel~Angel Virasoro.
    \newblock {\em Spin glass theory and beyond: An Introduction to the Replica
      Method and Its Applications}, volume~9.
    \newblock World Scientific Publishing Company, 1987.
    
    \bibitem{panchenko2013sherrington}
    Dmitry Panchenko.
    \newblock {\em The Sherrington-Kirkpatrick model}.
    \newblock Springer Science \& Business Media, 2013.
    
    \bibitem{parisi1995mean}
    Giorgio Parisi and Marc Potters.
    \newblock Mean-field equations for spin models with orthogonal interaction
      matrices.
    \newblock {\em Journal of Physics A: Mathematical and General}, 28(18):5267,
      1995.
    
    \bibitem{rangan2019vector}
    Sundeep Rangan, Philip Schniter, and Alyson~K Fletcher.
    \newblock Vector approximate message passing.
    \newblock {\em IEEE Transactions on Information Theory}, 65(10):6664--6684,
      2019.
    
    \bibitem{subag2018free}
    Eliran Subag.
    \newblock Free energy landscapes in spherical spin glasses.
    \newblock {\em arXiv preprint arXiv:1804.10576}, 2018.
    
    \bibitem{takahashi2020macroscopic}
    Takashi Takahashi and Yoshiyuki Kabashima.
    \newblock Macroscopic analysis of vector approximate message passing in a model
      mismatch setting.
    \newblock In {\em 2020 IEEE International Symposium on Information Theory
      (ISIT)}, pages 1403--1408. IEEE, 2020.
    
    \bibitem{takeda2006analysis}
    Koujin Takeda, Shinsuke Uda, and Yoshiyuki Kabashima.
    \newblock Analysis of cdma systems that are characterized by eigenvalue
      spectrum.
    \newblock {\em EPL (Europhysics Letters)}, 76(6):1193, 2006.
    
    \bibitem{talagrand2006free}
    Michel Talagrand.
    \newblock Free energy of the spherical mean field model.
    \newblock {\em Probability theory and related fields}, 134(3):339--382, 2006.
    
    \bibitem{tala2010libro}
    Michel Talagrand.
    \newblock {\em Mean field models for spin glasses: Volume I: Basic examples},
      volume~54.
    \newblock Springer Science \& Business Media, 2010.
    
    \bibitem{yan2014extension}
    Min Yan.
    \newblock Extension of convex function.
    \newblock {\em Journal of Convex Analysis}, 21(4):965--987, 2014.
    
    \bibitem{zdeborova2016statistical}
    Lenka Zdeborov{\'a} and Florent Krzakala.
    \newblock Statistical physics of inference: Thresholds and algorithms.
    \newblock {\em Advances in Physics}, 65(5):453--552, 2016.

\end{thebibliography}
\end{document}